\documentclass[11pt,openany,leqno]{article}
\usepackage{amsmath,amsthm,amsfonts,amssymb,amscd,url,color}
\usepackage{graphics}

\input{xy} \xyoption{all}
\usepackage[latin1]{inputenc}

\usepackage{enumerate}
\usepackage{hyperref}
\usepackage{epsfig} 
\input{xy} \xyoption{all}

\def\P{\mathbb I}
\def\R{\mathbb R}

\def\Z{\mathbb Z}

\def\sB{{\mathfrak B}}

\def\sb{{\mathfrak b}}

\newtheorem{prop}{Proposition}
\newtheorem*{theo*}{Theorem}

\newtheorem{defi}[prop] {Definition}
\newtheorem{lemm}[prop] {Lemma}

\newtheorem{cor}[prop]{Corollary}
\newtheorem{theo}{Theorem}

\newtheorem{question}[prop]{Question}

\theoremstyle{remark}
\newtheorem{exam}[prop]{Example}
\newtheorem{rema}[prop]{Remark}

\begin{document}

\title{Iterated Functions Systems, Blenders and Parablenders}

\author{Pierre Berger, \; Sylvain Crovisier, Enrique Pujals\thanks{This work is partially supported by the project BRNUH of Sorbonne Paris Cit\'e University and the French-Brazilian network.}}

\date{May 30, 2016}

\maketitle
\begin{abstract} 
We recast the notion of parablender introduced in~\cite{BE15} as a parametric  IFS. This is done using the concept of open covering property and looking to parametric IFS as systems acting on jets.
\end{abstract}

 Given a contractive IFS on $\R^n$, it is well known that for any point $x$
in the limit set and any admissible backward itinerary $\underline x$, it is possible to
consider its continuation for any IFS nearby.
In particular, if we consider parametric perturbations, then the
continuation of the point (provided an itinerary) is given by a smooth
curve $a\mapsto \underline x(a)$. In this note we study the $r$-jets of such a continuation,
(i.e. the $r$ first derivatives at the zero parameter).
As it is explained in section~\ref{s.paraIFS} (see also proposition~\ref{p.paraIFS})
the $r$-jet can be viewed as a point in the limit set  of a new multidimensional
contractive IFS acting on $\R^{n.(r+1)}$ .

For certain type of parametric IFS as the one introduced in~\cite{BE15} (as a
matter of fact in a more general setting), a further property holds:
\emph{ the $r$-jet
of any $C^r-$curve (with appropriate bounds in the derivative),
coincide with the $r$-jets of the continuation of some points of the IFS}.
We show that this property can be recast
as saying that the limit set of the IFS acting  on the jets has
interior. As we  want that this property remains valid for nearby
parametric families, we consider below IFS such that their actions on
$r$-jets exhibits a \emph{covering property} (see definition~\ref{d.covering} in section~\ref{s.IFS}).
In section~\ref{section3} we give an example of such a contractive IFS. The
example may seem extremely  restrictive but as it is shown in theorem A
their action on $r$-jets has the covering property.

In sections~\ref{s.blender} and~\ref{s.parablender},
we recall that the covering property for IFS is related to the notion of blender
for hyperbolic sets. Our purpose is to explain how the present approach for IFS can be use
to revisit the notion of parablender that was introduced in~\cite{BE15}.

\section{Iterated Functions Systems}\label{s.IFS}
\begin{defi}
A (contracting) \emph{Iterated Functions System (IFS)} is the data of a finite family $(f_\sb)_{\sb\in \sB}$ of contracting maps on $\R^n$.
The IFS is of class $C^r$, $r\ge 1$, if each $f_\sb$ is of class $C^r$. 
\end{defi}
The topology on the set of IFS of class $C^r$ (with $\operatorname{Card}\, \sB $ elements) is given by the product strong topology   $\prod_\sB C^r(\R^n,\R^n)$. 
The limit set of an IFS is:
\[\Lambda:= \{x\in \R^n: \exists (\sb_i)_i\in \sB^{\Z^-}, \; x= \lim_{k\to +\infty}  f_{\sb_{-1}}\circ \cdots \circ f_{\sb_{-k}} (0)\}.\]

The limit set $\Lambda$ is compact. One is usually interested in its geometry. Natural questions are: 
\begin{question}
Under which condition the limit set $\Lambda$ has non-empty interior?

Under which condition the limit set has $C^r$-robustly non empty interior ?  
\end{question}
Let us recall that a system satisfies a property \emph{$C^r$-robustly} if the property holds also for any $C^r$-perturbations of the system.
Both questions are still open, although there are already partial answers to them.  Let us state a classical sufficient property:
\begin{defi}\label{d.covering} The IFS $(f_\sb)_{\sb\in \sB}$ satisfies the \emph{covering property} if there exists a non-empty open set $U$ of $\R^n$ such that:
\[\operatorname{Closure}(U)\subset \bigcup_{\sb\in \sB} f_\sb(U).\]
\end{defi}
\begin{exam} For $\lambda \in (1/2,1)$, the IFS spanned by the two following one-dimensional maps
\[f_{1}\colon x\mapsto \lambda x+1,\]
\[f_{-1} \colon x\mapsto \lambda x-1,\]
satisfies the covering property since 
$[-2,2]\subset   \phi_{1}((-2,2)) \cup \phi_{-1}((-2,2))$.
\end{exam}

One easily proves the following:
\begin{prop} 
If the IFS $ (f_\sb)_{\sb\in \sB}$ satisfies the covering property with the open set $U$, then the limit set of the IFS contains $C^1$-robustly $U$.
\end{prop}
Hence the covering property is a sufficient condition for an IFS to have $C^r$-robustly non-empty interior, for every $r\ge 1$.  
\begin{question}\label{quesIFS} 
Is the covering property a necessary condition to  have 
$C^r$-robustly non-empty interior?
\end{question}
The answer to this question is not known even when $n=1$.  It is not clear to us that the answer would be independent of $r$. Indeed there are  phenomena which occur for $r>1$ and not for $r=1$, such as the stable intersection of regular cantor set \cite{GY01, Ne79, Gu11}.

Nevertheless the case $n=1$ and $Card\; \sB=2$ is simple. 
Indeed consider the IFS generated by two contracting maps $f_1,f_2$ of $\R$. Let $I$ be the convex hull of $\Lambda$: the endpoints of $I$ are the fixed points of $f_1$ and $f_2$. 
 If the interiors of $f_1(I)$ and $f_2(I)$ are disjoint, then a small perturbation makes $f_1(I)$ and $f_2(I)$ disjoint. Then it is easy to see that the IFS has empty interior. Otherwise  the interiors of $f_1(I)$ and $f_2(I)$ are not disjoint. They cannot coincide since the interior of $I$ is non empty and $f_1,f_2$ are contracting.
By removing the $\epsilon$-neighborhood of the endpoints of $I$,
one gets an interval $I_\epsilon$ which is covered by the images by $f_1$ and $f_2$ of its interior.
Consequently, if $\Lambda$ has robustly non empty interior it must satisfy the covering property.

\section{IFS with parameters}\label{s.paraIFS}
We denote $\mathbb I^k= [-1,1]^k$.
One considers $C^r$ parametrized families, i.e. elements in the
the Banach space $C^r(\mathbb I^k\times \R^n,\R^n)$, that we denote
$(f_a)_{a\in \mathbb I^k}$.

For each parameter $a_0\in  \mathbb I^k$, one introduces the jet space
$J^r_{a_0}(\mathbb I^k,\R^n)$, whose elements are the Taylor series
$(x_a, \partial_ax_a, \dots, \partial_a^rx_a)_{|a=a_0}$ at $a=a_0$ of $C^r$ functions $a\mapsto x_a$
in $C^r(\mathbb I^k,\R^n)$. Each $C^r$ family of maps  $(f_a)_{a\in \mathbb I^k}\in C^r(\mathbb I^k\times \R^n,\R^n)$ acts on the space of jets as a map $\widehat f$ defined by:
\[
\widehat f\; \colon \;
(x_a, \partial_ax_a, \dots, \partial_a^rx_a)_{|a=a_0}\mapsto 
 (f_a(x_a),\partial_a(f_a(x_a)),\dots,\partial^r_a (f_a(x_a)))_{|a=a_0}.\]

Our goal is to study parametrized IFS:
\begin{defi}
An \emph{Iterated Functions System (IFS) with parameter} is the data of a finite families $(f_{\sb, a})_{\sb\in \sB}$ of contracting maps on $\R^n$ depending on a parameter $a\in \mathbb I^k$.

The IFS with parameter is of class $C^r$, $r\ge 1$, if each $(f_{\sb, a})_{a\in \mathbb I^k}$ is in red $C^r(\mathbb I^k\times \R^n,\R^n)$. 
\end{defi}

For every $a\in \mathbb I^k$, we consider the limit set $\Lambda_a$ associated to the system $(f_{\sb,a})_{\sb\in \sB}$ equal to 
the set of points $X_a(\underline \sb):=  \lim_{k\to +\infty}  f_{\sb_{-1}, a}\circ \cdots \circ f_{\sb_{-k}, a} (0)$, among all $\underline \sb= (\sb_i)_i\in \sB^{\Z^-}$. These points admit a continuation
when $a$ varies: each function
$a\mapsto X_a(\underline \sb)$ is of class $C^r$.
One can consider its jet:
$$J^r_{a_0}X(\underline \sb):=(X_a(\underline \sb),\partial_a X_a(\underline \sb),\dots,
\partial^r_a X_a(\underline \sb))_{|a=a_0}.$$

Let us observe that the image of $X_a(\underline \sb)$  by $f_{\sb, a}$ is the continuation of $X_a(\sb\underline \sb)$ (where $\sb\underline \sb$ means a new sequence such that the first element now is $b$) i.e.: $f_{\sb, a}(X_a(\underline \sb))=X_a(\sb\underline \sb).$ In particular, its partial derivatives respect to the parameter at $a=a_0$ is nothing else that the image of $\widehat f_{\sb}$ on the jets of $X_a(\underline \sb)$; therefore, the jets of the continuations are the limit set of $(\widehat f_{\sb})_{\sb\in \sB}$. More precisely:
\begin{prop}\label{p.paraIFS}
The set $J^r_{a_0}\Lambda$ is the limit set of the IFS $(\widehat f_{\sb})_{\sb\in \sB}$
acting on the $r$-jet space $J^r_{a_0}\R^n$ which is generated by the finite collection of maps
$\{(f_{\sb, a})_{a\in \mathbb I^k},\,\sb\in \sB\}$ .
 \end{prop}
 It is natural to wonder if the following set has (robustly) non-empty interior:

\[J^r_{a_0}\Lambda := \left\{ J^r_{a_0}X(\underline \sb)\;:\; \underline \sb \in \sB^{\Z^-}\right\}\; .\]
We notice that if  $J^r_0\Lambda$ has non empty interior, then there exists a non-empty open subset $U\subset C^r(\mathbb I^k, \R^n)$ such that for every $(x_a)_{a\in \mathbb I^k}\in U$ there is $\underline \sb\in \sB^{\Z^-}$ satisfying:
\[ x_a= X_a(\underline \sb)+o(\|a\|^r)\; .\]
 
\section{The covering property for an affine IFS acting on the jet space} \label{section3}
In this note we study a simple IFS with parameter. We set $k=n=1$
and choose $r\geq 1$.
Let:
\[f_{+1, a}\colon x\mapsto \lambda_a x+1,\]
\[f_{-1, a} \colon x\mapsto \lambda_a x-1,\]
where $(\lambda_a)_{a\in\mathbb I}\in C^r(\mathbb I, (-1,1))$
satisfies $(\partial_a \lambda_a)_{|a=0} \not= 0$.
\begin{rema}
Any system in an open and dense set of
parametrized IFS generated by a pair of contracting affine maps with the same contraction,
can be conjugated to a system which coincides with $(f_{+1,a}, f_{-1,a})$ for $a$ close to $0$.
\end{rema}

 As $k=1$, after coordinate change on the parameter space, we can also assume that $\lambda_a= \lambda+a$ for $a$ in a neighborhood of $0$.  Note that the maps induced on $r$-jet space $J^r_0(\mathbb I,\R)$ are now:
\[\widehat f_{+1}\colon (x_a, \partial_ax_a, \dots, \partial_a^r x_a)_{|a=0}
\mapsto (\lambda x_a +1, \lambda \partial_a x_a + x_a , 
\dots, 
 \lambda \partial_a^r x_a + r \partial_a^{r-1} x_a)_{|a=0}.\]
 \[\widehat f_{-1}\colon (x_a, \partial_ax_a, \dots, \partial_a^r x_a)_{|a=0}
\mapsto (\lambda x_a-1, \lambda \partial_a x_a + x_a , 
\dots,
 \lambda \partial_a^r x_a + r \partial_a^{r-1} x_a)_{|a=0}.\]

In \cite{HS2014,HS2015}, its is proved that the IFS generated by $(\widehat f_{+ 1}, \widehat f_{-1})$ has non empty interior. Let us adapt their proof to obtain the following stronger result. 

\begin{theo}\label{paraIFS}
For any $r\geq1$, if $\lambda\in (0,1)$ is close enough to $1$,
then the IFS  generated by $(\widehat f_{+ 1}, \widehat f_{-1})$
acting on the $r$-jet space $J^r_0(\mathbb I,\R)$
satisfies the open covering property.
\end{theo}

\begin{cor}
Any IFS with parameter generated by two families of maps $C^r$-close to $(f_{+1,a})_a$ and
$(f_{-1,a})_a$ induces an IFS on the $r$-jet space $J^r_0(\mathbb I,\R)$
whose limit set has non-empty interior.
\end{cor}

\begin{proof}[Proof of theorem~\ref{paraIFS}]
Let us remark that the IFS  generated by $(\widehat f_{+ 1}, \widehat f_{-1})$ is conjugated (via affine coordinates change) to the IFS on $\R^{N}$, $N=r+1$,
generated by the maps
$$F_{+1}\colon X\mapsto J X+ T,\quad F_{-1}\colon X\mapsto J X- T,$$
where $T= (0,\dots, 0,1)$ and 

$$J=\begin{pmatrix}
\lambda & {N-1} & & & 0\\
 & \lambda & \dots & & \\
 & & \dots & \dots &\\
 & & & \lambda & 1\\
 0 & & & & \lambda
\end{pmatrix}.$$
One introduces a polynomial $P(x)=b_nx^n+b_{n-1}x^{n-1}+\dots+b_0$ with large degree $n$ which satisfies:
\begin{enumerate}
\item[i.] $b_0\neq 0$, $b_n=1$,
\item[ii.] $\sum_{j=0}^{n-1}|b_j|<2,$
\item[iii.] $P^{(i)}(1/\lambda)=0$ for $0\leq i\leq N-1$ (where $P^{(i)}(x)$ denotes the $i^\text{th}$
deviated polynomial of $P$),
\item[iv.] $P$ induces a projection $\pi\colon \R^n\to \R^N$ with rank $N$, defined by
$$\pi(u_{-n},\dots,u_{-1})=\bigg(\sum_{k=0}^{n-1}u_{k-n}.B_k^{(N-i)}(\lambda)\bigg)_{1\leq i\leq N}\quad \text{with} \quad
B_k(x)=\sum_{j=0}^kb_jx^{k-j}.$$
\end{enumerate}

\begin{prop}
For any $N\geq 1$, if $\lambda\in (0,1)$ is close enough to $1$,
there exists a polynomial $P$ satisfying conditions (i)--(iv).
\end{prop}
\begin{proof}
From \cite[Theorem 3.4]{HS2014}, there exists a monic polynomial
$Q(x)=x^n+a_{n-1}x^{n-1}+\dots+a_0$ such that $\sum_{i=0}^{n-1}|a_i|<2$
and $(x-1)^N|Q(x)$. Dividing by some $x^k$, one can assumes that $a_0\neq 0$.
One then sets $P(x)=\lambda^{-n}Q(\lambda\cdot x)$. Provided $\lambda$ is close enough to $1$,
it satisfies the conditions (i)--(iii). In order to check the last item, it is enough to check that
the following matrix has rank $N$:
$$\begin{pmatrix}
B_0 & B_1 & \cdots & B_{n-1}\\
B^{(1)}_0 & B^{(1)}_1 & \cdots & B^{(1)}_{n-1}\\
 \vdots &  \vdots & \ddots & \vdots &\\
B^{(N-1)}_0 & B^{(N-1)}_1 & \cdots & B^{(N-1)}_{n-1}
\end{pmatrix}.$$
This can be easily deduced from the fact that $B_0,B_1^{(1)},\dots,B_{N-1}^{(N-1)}$
are constant and non-zero polynomials and that $B_k^{(i)}=0$ when $k<i$.
\end{proof}

Let $S_{-1},S_1$ be the linear automorphisms of $\R^n$ defined for $\delta\in \{+1,-1\}$ by:
$$S_{\delta}\colon (u_{-n+1},\dots,u_{0})\mapsto (u_{-n},\dots,u_{-1})\quad \text{with} \quad
u_{-n}=\frac1{b_0} ( \delta-\sum_{j=1}^{n} b_{j} u_{j-n}).$$

\begin{prop}\label{p.sc}
$\pi$ is a semi-conjugacy:
$F_\delta\circ \pi=\pi\circ S_{\delta}$.
\end{prop}
Before proving the proposition, one checks easily the following relations.
\begin{lemm}\label{l1} If $1\leq k\leq n$ and $i\geq 1$,
$$B_k(x)=x\cdot B_{k-1}(x)+b_k \quad \text{and} \quad
B_k^{(i)}(x)=x\cdot B_{k-1}^{(i)}(x)+i\cdot B_{k-1}^{(i-1)}(x).$$
\end{lemm}

Since $B_n(x)=x^n\cdot P(1/x)$ and $P^{(i)}(1/\lambda)=0$ for $0\leq i\leq N-1$ one gets
\begin{lemm}\label{l2} If $0\leq i\leq N-1$,
$$B^{(i)}_n(\lambda)=0.$$
\end{lemm}

\begin{proof}[Proof of the Proposition]
One has to check $F_\delta\circ \pi \circ S_{\delta}^{-1}=\pi$.

One fixes $(u_{-n},\dots,u_{-1})$.
It is sent by $S_{\delta}^{-1}$ to $(u_{-n+1},\dots,u_0)$ with
$u_0=\delta-\sum_{j=0}^{n-1}b_ju_{j-n}.$
Then by $\pi$ to
$\bigg(\sum_{k=0}^{n-1}u_{k+1-n}\cdot B_k^{(N-i)}(\lambda)\bigg)_{1\leq i\leq N}$.
Applying $F_{\delta}$, one gets a vector $(v_1,\dots,v_N)$ whose $i^\text{th}$ coordinate coincides with

$$v_i=\lambda \sum_{k=0}^{n-1}u_{k+1-n}\cdot B_k^{(N-i)}(\lambda)+(N-i)\sum_{k=0}^{n-1}u_{k+1-n}\cdot B_k^{(N-i-1)}(\lambda)\quad
\text{if} \quad i\neq N,$$
$$v_N=\lambda \sum_{k=0}^{n-1}u_{k+1-n}\cdot B_k(\lambda)+\delta\quad
\text{otherwise}.$$
For the $N-1$ first coordinates, from lemma~\ref{l1} one gets
$$v_i=\lambda\cdot \sum_{k=1}^{n}u_{k-n}\cdot B_{k-1}^{(N-i)}(\lambda)+(N-i)\sum_{k=1}^{n}u_{k-n}\cdot B_{k-1}^{(N-i-1)}(\lambda)=\sum_{k=1}^{n}u_{k-n}\cdot B_{k}^{(N-i)}(\lambda)$$
$$= \sum_{k=0}^{n-1}u_{k-n}\cdot B_{k}^{(N-i)}(\lambda) \quad\quad\quad \quad\quad\quad \text{since $B_0^{(N-i)}(\lambda)=B_n^{(N-i)}(\lambda)=0$}.$$
For the last coordinate, one gets similarly from lemmas~\ref{l1} and~\ref{l2}
$$v_N=\lambda\cdot \sum_{k=1}^{n}u_{k-n}\cdot B_{k-1}(\lambda)+\delta=\lambda\cdot \sum_{k=1}^{n}u_{k-n}\cdot B_{k-1}(\lambda)+\sum_{k=0}^n u_{k-n}\cdot b_k=
\sum_{k=1}^{n}u_{k-n}\cdot B_{k}(\lambda)+u_{-n}\cdot b_0$$
$$= \sum_{k=0}^{n-1}u_{k-n}\cdot B_{k}(\lambda) \quad\quad\quad\quad\quad\quad\quad\quad \text{since $B_0(\lambda)=b_0$ and $B_n(\lambda)=0$}.$$
This gives $(v_1,\dots,v_N)=\pi(u_{-n},\dots,u_{-1})$ as required.
\end{proof}
\bigskip

Since $\sum_{j=0}^{n-1}|b_j|<2$, one can choose
$\eta>1$ such that ${ \eta^n} \sum_{j=0}^{n-1}|b_j| <\eta+1$ and let $A$ be the image by $\pi$
 of
$$\Delta:=(-{\eta^n},{ \eta^n})\times (-{ \eta^{n-1}},{\eta^{n-1}})\times\dots\times(-\eta,\eta).$$

\begin{prop} The subset $A$ is open and satisfies: $\operatorname{Closure}(A)\subset F_{-1}(A)\cup F_{1}(A)$.
\end{prop}
\begin{proof}
The linear map $\pi$ is open since it has rank $N$.
Since $\pi$ sends compact sets to compact sets, it is enough to prove
$$\pi(\operatorname{Closure}(\Delta))\subset F_{+1}\circ \pi (\Delta)\cup
F_{-1}\circ \pi (\Delta).$$
By proposition~\ref{p.sc}, one has to check the following inclusion:
$$\operatorname{Closure}(\Delta)\subset S_{+1}(\Delta)\cup
S_{-1}(\Delta).$$

Consider any point $(u_{-n},\dots,u_{-1})$ in $\operatorname{Closure}(\Delta)$.
By our choice of $\eta$ and since $|u_j|\le{ \eta^n}$ for each $-n\leq j\leq -1$,
there exists $u_0\in (-\eta,\eta)$ and $\delta\in \{-1,1\}$ satisfying the relation
$u_0=\delta-\sum_{j=0}^{n-1}b_ju_{j-n}$.
{ Since $|u_i|\leq \eta^i$ we get $|u_{i-1}|<\eta^{i}$.}
One deduces that  $(u_{-n+1},\dots,u_{-1},u_0)$ belongs to
$\Delta$. Since $b_n=1$,
one has $\sum_{j=0}^{n}b_ju_{j-n}=\delta$, so
 $S_{\delta}(u_{-n+1},\dots,u_{-1},u_0)=(u_{-n},\dots,u_{-1})$.
This proves the required inclusion.
\end{proof}
The covering property is thus satisfied and the Theorem is proved.
\end{proof}

\section{Blenders for endomorphisms}\label{s.blender}
Our motivations for studying the action of IFS on jet spaces come from hyperbolic differentiable dynamics, and more specifically from the study of blenders and para-blenders that we explain in these two last sections.

If $f\colon M\to M$ is a $C^1$-map on a manifold $M$,
a compact subset $K\subset M$ is \emph{hyperbolic} if:
\begin{itemize}
\item[--] $f$ is a local diffeomorphism on a neighborhood of $K$,
\item[--] $K$ is  invariant (i.e. $f(K)= K$),
\item[--] there exists an invariant sub-bundle $E^s\subset TM_{|K}$ and $N\geq1$ so that $\forall x\in K$:
\end{itemize}

\[ D_xf(E_x^s)\subset E^s_{f(x)},\quad \| D_xf^N|E^s_x\| <1,\quad
{\| p_{E^s_\bot}\circ (D_xf^N)^{-1}|E^s_{\bot,x}\|<1}, \]
where $E^s_{\bot,x}$ is the orthogonal complement of $E^s_x$ and $p_{E^s_\bot}$ the orthogonal projection onto it.

Note that the map $f$ is in general not invertible.
Hence one can define an unstable space at any $x\in K$,
but it is in general not unique: it depends on the choice of a preorbit of $x$.
We recall that the inverse limit $\overleftarrow K$ is the set of preorbits:
\[\overleftarrow K:= \{(x_i)_{i\le 0}\in K^{\Z^-} :\; f(x_i)=x_{i+1},\; \forall i<0\}.\] 
The map $f$ induces a map $\overleftarrow f\colon (x_i)\mapsto (f(x_{i}))$ on $\overleftarrow K$.

For every preorbit $\underline x =(x_i)_{i\le 0}\in \overleftarrow K$ and for every $\epsilon>0$ small enough, the following set is a submanifold of dimension $\operatorname{Codim}(E^s)$:
\[W^u(\underline x, \epsilon)=\{x'\in M:\exists \underline x'\in \overleftarrow K \text{ s.t. }
x'_0= x',\; \forall i\;d(x'_i ,x_i)<\epsilon \text{ and } \lim_{i\to -\infty} d(x'_i ,x_i)=0\},\]
and is called \emph{local unstable manifold}
(also denoted by $W^u(\underline x, \epsilon,f)$ when one specifies the map $f$).

We recall that $K$ is \emph{inverse-limit stable}: for every $C^1$-perturbation $f'$ of $f$, there exists a unique map $\pi_{f'}\colon \overleftarrow K\to M$ which is $C^0$ close to the zero coordinate projection $\pi \colon (x_i)_i \in \overleftarrow K \mapsto x_0\in M$ so that the following diagram commutes:
\[f'\circ \pi_{f'}=\pi_{f'}\circ \overleftarrow f\; .\] 
Moreover $\pi_{f'}(\overleftarrow K)$ is hyperbolic for $f'$ and is called the \emph{hyperbolic continuation of $K$ for $f'$}.
In particular any $\underline x$ has a continuation, that is the sequence $\underline x'=(x'_i)$
in $\pi_{f'}(\overleftarrow K)$ such that
$$x'_i=\pi_{f'}((x_{i+k})_{k\leq i}).$$
The local unstable manifold of $\underline x'$ will be denoted
$W^u(\underline x, \epsilon, f')$. When $\epsilon$ is implicit,
the local unstable manifolds are also denoted by $W^u_{loc}(\underline x)$ and $W^u_{loc}(\underline x, f')$.
\medskip

The notion of blender was first introduced in the invertible setting by \cite{BD96} to construct robustly transitive diffeomorphisms, and then \cite{BD99, DNP} to construct locally generic diffeomorphism with infinitely many sinks. The work \cite{BE15} deals with blenders for endomorphisms.
\begin{defi} A \emph{$C^r$-blender} for a $C^r$-endomorphism is a hyperbolic set $K$ such that the union of its local unstable manifolds has $C^r$-robustly a non-empty interior:
there exists a non-empty open set $U\subset M$ which is contained in the
union of the local unstable manifolds of the hyperbolic continuation of $K$ for any
endomorphism $f'$ $C^r$-close to $f$.
\end{defi}
The classical definition for diffeomorphisms is more general:
fixing an integer $d$ smaller than the stable dimension of $K$,
it asserts that there exists an open collection $U$ of embeddings of the $d$-dimensional disc
in $M$ such that any $D\in U$ intersects the union of the local unstable manifolds of the hyperbolic continuation of $K$ for any diffeomorphism $f'$ $C^r$-close to $f$.
\medskip

\begin{exam}\label{e.blender} For $\lambda \in (1/2,1)$, we consider a local diffeomorphism $f$ of $\R^2$ whose restriction to 
$([-2,-1]\cup[1,2])\times [-1/(1-\lambda),1/(1-\lambda)]$ is:
\[(x,y)\mapsto (4 |x|-6, \lambda y+\operatorname{sgn}(x)),\]
where $\operatorname{sgn}(x)$ is equal to $\pm 1$ following the sign of $x$.
The set of points $(x,y)$ whose iterates are all contained in
$([-2,-1]\cup[1,2])\times [-1/(1-\lambda),1/(1-\lambda)]$ is a hyperbolic set $K$
which is a $C^1$-blender.
\begin{proof}
Note that $K$ is locally maximal: any orbit $(x_n,y_n)_{n\in \Z}$ contained in a
small neighborhood of $([-2,-1]\cup[1,2])\times [-1/(1-\lambda),1/(1-\lambda)]$ belongs to $K$. For diffeomorphisms $C^1$-close,
such an orbit is contained in the hyperbolic continuation of $K$.

For $\eta>0$ small, let
$\Delta_\eta:= ([-2-\eta,-1+\eta]\cup [1-\eta, 2+\eta])\times [-2,2]$.
For every $C^1$-perturbation $f'$ of $f$, it holds:
\[f'(\operatorname{Interior}(\Delta_\eta))\supset [-2-\eta,2+\eta]\times [-2,2]\; .\] 
Hence every point $(x,y)\in [-2,2]\times [-2,2]$ admits an $f'$-preorbit $(\underline x,\underline y)=(x_n,y_n)_{n< 0}$
in $\Delta_\eta$. It shadows a unique preorbit $\underline z\in \overleftarrow K$.
Consequently we have $(x,y)\in W^u_{loc}(\underline z, f')$
and this shows that $K$ is a $C^1$-blender.
\end{proof}
\end{exam}
\bigskip

More generally, given a finite set $\sB$, one can construct 
disjoint intervals $\sqcup_{\sb\in \sB} I_\sb\subset [-1,1]$ and an expanding map $q\colon \sqcup_{\sb\in \sB} I_\sb\to [-1,1]$ so that $q(I_\sb)$ is equal to $[-1,1]$.

Then given an  IFS $(f_\sb)_{\sb\in \sB}$ by contracting diffeomorphisms of $\R^n$, we can define
a map:
\[f\colon (x,y)\in \sqcup_{\sb\in \sB} I_\sb\times \R^n\mapsto (q(x), f_\sb(y)),\quad \text{if } x\in I_\sb\; .\]
whose maximal invariant set $K$ is hyperbolic. The second coordinate projection of $K$ is the limit set $\Lambda$ of the IFS. Also  if the IFS satisfies the covering property, then $K$ is a blender (see also~\cite{BKR14}).
\medskip

Despite its fundamental aspect, our specific interest for question \ref{quesIFS} is to know whereas a covering-like property is equivalent to the above definition of blender.

\section{Parablenders}\label{s.parablender}

In this section we deal with $C^r$-families of endomorphisms of a compact manifold $M$, with $k\geq 1$ parameters, that is
elements in { $C^r(\mathbb I^k\times M,M)$}, denoted as $(f_a)_{a\in \mathbb I^k}$.

Let $f$ be a local diffeomorphism on $M$ with a hyperbolic set $K$.
Given { a $C^r$-family} $(f_a)_{a\in \P^k}$ $C^0$-close to the constant family $(f)_{a\in \P^k}$, for each $\underline x\in \overleftarrow K$ we can consider the family of local unstable manifolds { $(W^u_{loc}(\underline x,f_a))_{a\in \P^k}$}. It is easy to see that this family is $C^0$-close to {
the constant family} $(W^u_{loc}(\underline x, f))_{a\in \P^k}$. Actually  it is much more:
{ \begin{prop}\label{prop16}
For every $\underline x\in \overleftarrow K$, the set $\cup_{a\in \P^k }  \{a\}\times W^u_{loc}(\underline x,f_a)$ is a $C^r$-submanifold of $\P^k\times M$ which depends continuously on $\underline x$. In other words, the family of submanifolds $(W^u_{loc}(\underline x,f_a))_{a\in \P^k}$ is of class $C^r$ and depends continuously on $\underline x\in \overleftarrow K$ for the $C^r$-topology.
\end{prop}
\begin{proof}
The submanifolds $\{(a, \pi_{f_a}(\underline x)): \; a \in \P^k\}$ among $\underline x\in \overleftarrow K$ form the leaves of lamination immersed in $\P^k\times M$. 
The dynamics $(a,x)\mapsto(a, f_a(x))$ leaves invariant this lamination, and is $r$-normally hyperbolic at it. 
By Proposition 9.1 of \cite{BEsbm}, the local unstable set of each of these leaves is a  $C^r$-submanifold which depends continuously on $\underline x\in \overleftarrow K$. 
\end{proof}}
We are now able to state:
 \begin{defi}  A \emph{$C^r$-parablender} at  $a_0\in \P^k$ for a family of endomorphisms
 { $(f_a)_a\in C^r(\P^k\times M,M)$} is a hyperbolic set $K$ for $f_{a_0}$ such that:
\begin{itemize}
\item[--] for every $\gamma$ in a non-empty open subset $U$ of $C^r(\P^k,M)$,
\item[--] and every $(f'_a)_a$ in a neighborhood $V$ of  $(f_a)_a$ in { $C^r(\P^k\times M,M)$},
\end{itemize}
there exist  $\underline x\in \overleftarrow K$ and $\zeta\in C^r(\P^k,M)$ satisfying:
\begin{itemize}
\item[--] $\zeta(a)$ belongs to $W^u_{loc}(\pi_{f_a}(\underline x),f_a)$ for every $a\in \P^k$,
\item[--] the $r$-first derivatives of $\gamma$ and $\zeta$ are equal at $a_0$:
\[\zeta(a_0)= \gamma(a_0),\quad   D\zeta(a_0)= D\gamma(a_0),\quad \dots, \quad D^r\zeta(a_0)= D^r\gamma(a_0)\; .\]
\end{itemize}
\end{defi}
In particular $K$ is a $C^r$-blender.
\medskip

The concept of parablender was introduced in \cite{BE15} to prove that the diffeomorphisms with finitely many attractors are \emph{not typical} in the sense of Kolmogorov, a result in the opposite direction to a conjecture of Pugh-Shub \cite{PS95} and to the main conjecture of Palis \cite{Pa00}. 
The parablenders defined therein are based on IFS of $\R$ generated by 
$2^{\dim\, \{P\in \R[X_1,\dots ,X_{k}]:\; \deg(P)\le r\}}$ elements.

We give here a new example in the case $k=1$ and $n=2$, based on the above IFS theory.
Note that the number of elements is reduced to $2$ and is now independent from the smoothness $r$.

\begin{theo}\label{thm.parblender}
For any surface $M$ and any $r\geq 1$,
there exists a family $(F_a)_a\in { C^r(\mathbb I\times M,M)}$
which admits a $C^r$-parablender induced by an IFS with $2$ elements.
\end{theo}
\begin{proof}
The construction is realized inside a disc and can be extended to any surface $M$.
It uses the previous examples:
for $\lambda\in (1/2, 1)$, we consider a family  $(F_a)_{a\in \P}\in { C^r(\mathbb I\times \R^2,\R^2)}$ whose restrictions to 
$([-2,-1]\cup[1,2])\times [-1/(1-\lambda),1/(1-\lambda)]$ is:
\[F_a\colon (x,y)\mapsto (4 |x|-6, (\lambda+a) y+\operatorname{sgn}(x)).\]
The set of orbits $(x_n,y_n)_{n\in \Z}$ of $F_0$
that are contained in $([-2,-1]\cup[1,2])\times [-1/(1-\lambda),1/(1-\lambda)]$
project through the map $(x_n,y_n)_n\mapsto (x_0,y_0)$ on a hyperbolic set $K$.

Denoting $f_{\pm1,a}\colon y\mapsto (\lambda+a) y\pm 1$ the maps introduced in section~\ref{section3},
and $g_{\pm1}\colon x\mapsto \pm 4x-6$, we get:
\[F_a\colon (x,y)\mapsto (g_{\operatorname{sgn}(x)}(x), f_{\operatorname{sgn}(x),a}(y).\]

The families $(f_{\pm1,a})_{a\in \mathbb I}$ induce an IFS on the jet space $J_0^r(\P,\R)$
generated by two maps $(\widehat f_{\pm 1})$. 
Theorem \ref{paraIFS} states that (for $\lambda<1$ close enough to $1$)
there exists a non-empty open set $A\subset J_0^r(\P,\R)$ such that $\widehat f_{+1}( A)\cup \widehat f_{-1}( A)$ contains the closure of $A$.
 Let $\delta>0$ be the Lebesgue number of this covering: every point in $\operatorname{Closure}(A)$ is the center of a { closed } $\delta $- ball contained in   $\widehat f_{+1}( A)$ or in  $\widehat f_{-1}( A)$. 
 Let $A_+$ and $A_-$ be the subsets of $A$ formed by points whose $\delta$-neighborhoods are contained in  respectively $\widehat f_{+1}( A)$ and $\widehat f_{-1}( A)$.  
Note that $A_+$ and $A_-$ are open sets and:
\[A = A_+\cup A_-\; \text{and}\; 
\operatorname{Closure}( \widehat f_{+1}^{-1}(A_+)\cup \widehat f_{-1}^{-1}(A_-))\subset A \;.\]

On the other hand, $g_{+1}$ and $g_{-1}$ act on the $C^r$-jet space $J_0^r(\P,\R)\approx \R^{r+1}$ as maps
$$\widehat g_{\pm1}\colon (\partial^i_a x_a)\mapsto (g_{\pm1}(x_a), \pm 4 \partial_a x_a, \dots, \pm 4 \partial^r_a x_a).$$
 Let $B$ be the open subset of $J_0^r(\P,\R)$ equal to:
\[B:= (-2-\eta,2+\eta)\times (-\eta,\eta)^r\; ,\]
for $\eta>0$ small enough so that each of the inverse maps $\widehat  g_{\pm1}^{-1}$ sends the closure of $B$ into $B$.

We notice that the action $\widehat F$ of $(F_a)_a$ on the $r$-jets $J^r_0(\P,\R^2)$ has two inverse branches:  $\widehat F^{-1}_{-1}:=(\widehat  g_{-1}^{-1},\widehat f_{-1}^{-1})$ and $\widehat F^{-1}_{+1}:=(\widehat  g_{+1}^{-1},\widehat f_{+1}^{-1})$ satisfying with the open subsets $W_{\pm 1}:=B\times A_{\pm 1}$ and $W:=B\times A=W_{+1}\cup W_{-1}$ of $J^r_0(\P,\R^2)$ the following:
 \begin{equation}\tag{$\star$}\operatorname{Closure} (\widehat F^{-1}_{+1}(W_{+1})\cup \widehat F^{-1}_{-1}(W_{-1}))\subset W=W_{+1}\cup W_{-1}\; .
 \end{equation}
For an open set $V$ of $C^r$-perturbations $(F'_a)_a$ of $(F_a)_a$, the inverse of branches $\widehat F'^{-1}_{+1}$ and $\widehat F'^{-1}_{-1}$ of the induced action on the $r$-jets still satisfy ($\star$).
\medskip

Let $U$ be the non-empty open set of curves $a\mapsto \gamma(a)\in C^r(\P, \R^2)$ so that the $r$-jet
$\widehat \gamma:=(\gamma,\partial_a\gamma,\dots,\partial^r_a\gamma)_{|a=0}$ of $\gamma$ at $a=0$ lies in $W$.  By the latter inclusion $(\star)$, we can define inductively a sequence
$\underline \delta := (\delta_i)_{i\le 0}\in \{-1,+1\}^{\Z^-}$ 
and
$(\widehat \gamma_i)_{i\le 0}\in \prod_{i\le 0} W_{\delta_i}$  so that $\widehat \gamma_0 =\widehat \gamma$, and for $i\le 0$,  $\widehat \gamma_{i} = \widehat F'_{\delta_{i}}(\widehat \gamma_{i-1})$.  Note that given a $\widehat \gamma$, the sequences  $\underline \delta$ and $(\widehat \gamma_i)_{i\le 0}$ are in general not uniquely defined. 
We remark that $\widehat \gamma_i$ is the $r$-jet at $a=0$ of the curve  $a\mapsto \gamma_i(a)$
defined by
\[\gamma_i(a) := (F'_{a}|Y_{\delta_i})^{-1}
\circ \cdots \circ 
(F'_{a}|Y_{\delta_0})^{-1}(\gamma(a)),\]
\[\text{where } Y_{+1}:=[1-\eta, 2+\eta]\times \left[\frac{-1-\eta}{1-\lambda},\frac{1+\eta}{1-\lambda}\right]
\text{ and }
Y_{-1}:=[-2-\eta, -1+\eta]\times \left[\frac{-1-\eta}{1-\lambda},\frac{1+\eta}{1-\lambda}\right]  .\]

The sequence $\underline \delta$ defines a local unstable manifold of $K$
\[W^u_{loc}(\underline \delta; F_0):=  \bigcap_{j\ge 0} F^{j+1}_0(Y_{\delta_j}).\]
It admits a continuation $W^u_{loc}(\underline \delta; F'_a)$ for any family $(F'_a)_a$ close to $(F_a)_a$
and any parameter $a$ close to $0$. This unstable manifold also contains
the projection of a point $\underline x\in \overleftarrow K$ so that
$W^u_{loc}(\underline \delta; F'_a)={ W^u_{loc}(\underline x, F'_a)}$.

Let $\zeta(a)$ be the vertical projection of $\gamma(a)$ into $W^u(\underline \delta; f_a)$ for every $a\in \P$. As $(W^u(\underline \delta; f_a))_a$ is of class $C^r$ by proposition \ref{prop16}, the function $a\mapsto \zeta(a)$ is of class $C^r$. 
We now consider the vertical segment $C(a):= [\zeta(a),\gamma(a)]$. Up to shrinking slightly $V$, we can assume that the stable cone field $\mathcal{C}:= \{(u,v)\colon |u|\le \eta |v|\}$ is backward invariant by each $F'_a$ and $(\lambda+\eta)$-contracted by $DF'_a$, for every $a$ close to $0$. Thus the curve 
$$C_i(a) := (F'_a| Y_{\delta_{i}})^{-1}\circ \cdots \circ (F'_a | Y_{\delta_0})^{-1}(C(a))$$  
has its tangent space in $\mathcal{C}$ and connects $\gamma^i(a)$ to the local unstable manifold  { $W^u_{loc}\underline x, F'_a)$}. 

By proposition \ref{prop16}, the $r$ first derivatives of { $(W^u_{loc}(\underline x, F'_a))_a$}
are uniformly bounded.
By assumption $(\gamma_i(a))_i$ has its $r$-first derivatives for $a$ small enough contained in $W$,
hence uniformly bounded.
Thus there exists $A>0$ independent of $i$ and so that for any
$a$ small enough (depending on $i$),  the length of $C_i(a)$ is at most $ A\sum_{j=0}^r |a|^j+|a|^{r}\rho_i(a)$, with $\rho_i$ a continuous function equal to $0$ at $a=0$.  

We recall that $DF'_a|\mathcal{C} $ is $(\lambda+\eta)$-contracting. As $C_i(a)$ has its tangent space in $\mathcal{C}$ it comes that the length of $C(a)$ is at most $(\lambda+\eta)^{|i|} [A\sum_{j=0}^r |a|^j+|a|^r\rho_i(a)]$ for every $i\le 0$ and $a$ small enough in function of  $i$. 
 This proves that  the length of $C(a)$ and its $r$ first derivative w.r.t. $a$
 at $a=0$ are smaller than $(\lambda+\eta)^{|i|}$, for every $i\le 0$.  
Hence they vanish all and so the $r$-first derivatives of $\zeta$ and $\gamma$ are equal at $a=0$. \end{proof}

{
\section{Nearly affine blenders and parablenders}

The previous constructions may be realized in the following more general setting.
Let us fix $r\geq 1$ and choose $\lambda<1$ close to $1$.

\begin{defi}
For $\varepsilon>0$ small, we say that
a local diffeomorphism defined on a neighborhood of the rectangle
$R:=[-2,2]\times [-1/(1-\lambda), 1/(1-\lambda)]$
is a $\varepsilon$-\emph{nearly affine blender with contraction $\lambda$} if
there exists two inverse branches $g_+,g_-$ for $f^{-1}$:
\begin{itemize}
\item[--] $g_+$ is defined on a neighborhood of
$Y_+:=[-2,2]\times [(1-2\lambda)/(1-\lambda),1/(1-\lambda)]$ and is $\varepsilon$-close to the map
$(x,y)\mapsto (0,(y-1)/\lambda)$ for the $C^1$-topology;
\item[--] $g_-$ is well defined on a neighborhood of
$Y_-:=[-2,2]\times [-1/(1-\lambda),(2\lambda-1)/(1-\lambda)]$ and is $\varepsilon$-close to the map
$(x,y)\mapsto (0,(y+1)/\lambda)$ for the $C^1$-topology;
\item[--]
$g_+([-2,2]\times \{(1-2\lambda)/(1-\lambda),1/(1-\lambda)\})$ and
$g_-([-2,2]\times \{-1/(1-\lambda),(2\lambda-1)/(1-\lambda)\})$ are disjoint from $R$.
\end{itemize}
\end{defi}
The last item implies that there exist two maps $\psi_-<\psi_+:=[-2,2]\to [-1/(1-\lambda), 1/(1-\lambda)]$
whose graphs are contained in $Y_-$ and $Y_+$ respectively and contracted by the respective
branches of $f^{-1}$, so that the strip of $[-2,2]\to [-1/(1-\lambda), 1/(1-\lambda)]$ bounded by these two graphs
is contained in its image. Arguing as in the example~\ref{e.blender}, one shows that the maximal invariant
set in $R$ is a blender.
\begin{figure}[ht]
\begin{center}
\includegraphics[scale=0.4]{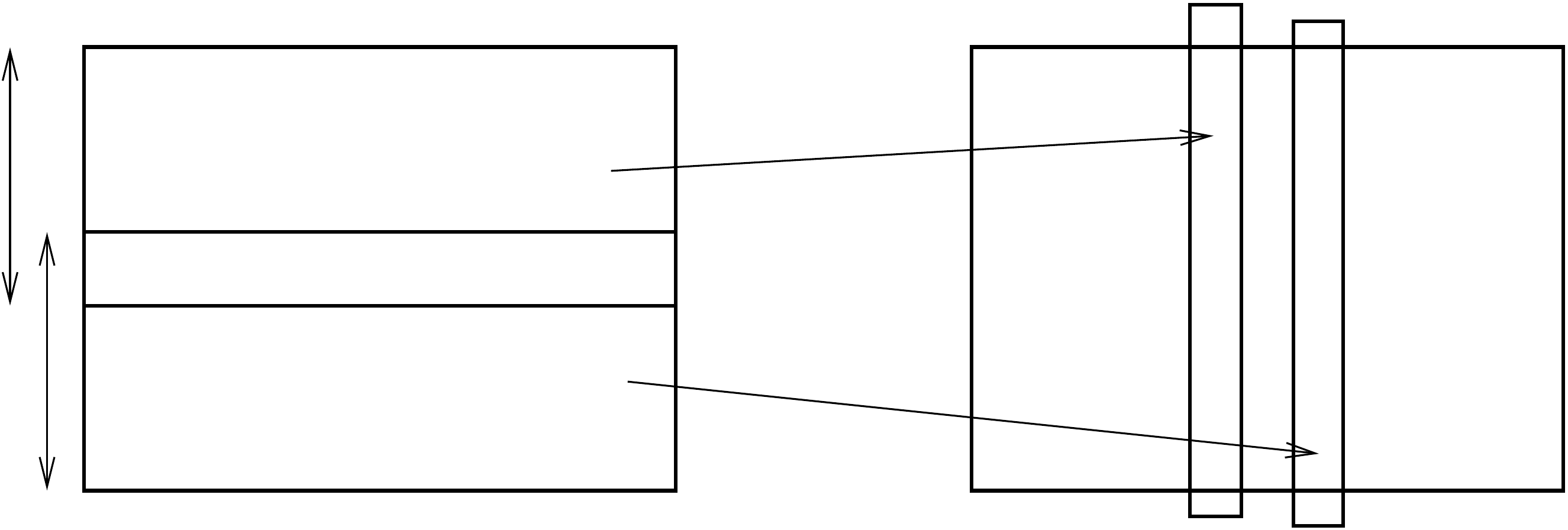}
\begin{picture}(0,0)
\put(-325,65){$Y_+$}
\put(-325,30){$Y_-$}
\put(-145,30){$g_-$}
\put(-145,80){$g_+$}
\end{picture}
\end{center}\end{figure}

\begin{defi}
For $\varepsilon>0$ small, we say that
a $C^r$-family $(f_a)_{a\in \mathbb I}$ is a $\varepsilon$-\emph{nearly affine parablender with contraction $\lambda$
and $a=0$} if:
\begin{itemize}
\item[--] $f_0$ is a $\varepsilon$-\emph{nearly affine blender with contraction $\lambda$};
\item[--]  for some $\alpha>0$, the family $(f^{-1}_a)_{|a|\leq \alpha}$ on a neighborhood of $Y_+$
is $\varepsilon$-close in $C^r$-topology to
$$(a,x,y)\mapsto (0,(y-1)/(\lambda+a));$$
\item[--]  for some $\alpha>0$, the family $(f^{-1}_a)_{|a|\leq \alpha}$ on a neighborhood of $Y_-$
is $\varepsilon$-close in $C^r$-topology to
$$(a,x,y)\mapsto (0,(y+1)/(\lambda+a)).$$
\end{itemize}
\end{defi}
The same proof as for theorem~\ref{thm.parblender}
shows that the maximal invariant set for $f_0$ in $R$ is a $C^r$-parablender for $(f_a)_{a\in \mathbb I}$
at $a=0$ provided $\lambda$ has been chose close enough to $1$ in function of $r$ and provided
$\varepsilon>0$ has been chosen small enough in function of $r$ and $\lambda$.
}

\bibliographystyle{alpha}
\def\cprime{$'$} \def\cprime{$'$} \def\cprime{$'$}

\bigskip

\noindent \emph{Pierre Berger}, {\small LAGA,  CNRS - UMR 7539, Universit\'e Paris 13, 93430 Villetaneuse, France.}

\vskip 3pt

\noindent \emph{Sylvain Crovisier}, {\small LMO, CNRS - UMR 8628, Universit\'e Paris-Sud 11, 91405 Orsay, France.}

\vskip 3pt

\noindent
\emph{Enrique Pujals}, {\small IMPA,
Estrada Dona Castorina 110, 22460-320 Rio de Janeiro, Brazil.}

\end{document}